\documentclass{emsprocart}

\contact[olivier.guichard@math.unistra.fr]{Olivier Guichard, Institut de Recherche Math\'ematique Avanc\'ee, UMR 7501\\
 Universit\'e de Strasbourg et CNRS\\
7 rue Ren\'e-Descartes, 67000 Strasbourg, France}

\contact[wienhard@uni-heidelberg.de, ]{Anna Wienhard, Ruprecht-Karls Universit\"at Heidelberg, Mathematisches Institut, Im Neuenheimer Feld~205, 69120 Heidelberg, Germany
\newline HITS gGmbH, Heidelberg Institute for Theoretical Studies, Schloss-Wolfsbrun\-nenweg 35, 69118 Heidelberg, Germany}

\newtheorem{theorem}{Theorem}[section]
\newtheorem{corollary}[theorem]{Corollary}

\newtheorem{conjecture}[theorem]{Conjecture}

\theoremstyle{definition}
\newtheorem{definition}[theorem]{Definition}
\newtheorem{remark}[theorem]{Remark}
\newtheorem{examples}[theorem]{Examples}

\newcommand{\RR}{\mathbb{R}}
\newcommand{\CC}{\mathbb{C}}
\newcommand{\HH}{\mathbb{H}}
\newcommand{\OO}{\mathbb{O}}
\newcommand{\ZZ}{\mathbb{Z}}

\newcommand{\PP}{\mathbb{P}}
\newcommand{\SL}{\mathrm {SL}}
\newcommand{\GL}{\mathrm {GL}}
\newcommand{\Sp}{\mathrm {Sp}}
\newcommand{\SO}{\mathrm {SO}}

\title[Positivity and higher Teichm\"uller theory]{Positivity and higher Teichm\"uller theory}

\author[Olivier Guichard, Anna Wienhard]{Olivier Guichard, Anna Wienhard \thanks{The second author acknowledges support by the National Science Foundation under agreement DMS-1536017, by the Sloan Foundation, by the Deutsche Forschungsgemeinschaft, by the European Research Council under ERC-Consolidator grant 614733, and by the Klaus-Tschira-Foundation.}}

\begin{document}

\begin{abstract}
We introduce $\Theta$-positivity, a new notion of positivity in real semisimple Lie groups. The notion of $\Theta$-positivity generalizes 
at the same time Lusztig's total positivity in split real Lie groups 
as well as well known concepts of positivity in Lie groups of Hermitian type. We show that there are two other families of Lie groups, $\SO(p,q)$ for $p\neq q$ 
and a family of exceptional Lie groups, which admit a $\Theta$-positive structure. We describe key aspects of $\Theta$-positivity and make a connection with 
representations of surface groups and higher Teichm\"uller theory. 
\end{abstract}



\maketitle
\section{Introduction}

A totally positive matrix in $\GL(n,\RR)$ is a matrix all of whose minors are positive. Totally positive matrices arose first in work of Schoenberg \cite{Schoenberg} and Gantmacher and Krein \cite{Gantmacher_Krein}, and have since become very important in a wide array of mathematical fields, ranging from stochastic processes to representation theory. 

In 1994 Lusztig \cite{Lusztig} generalized  total positivity to the context of general split real semisimple reductive Lie groups. It plays an important role in representation theory with many interesting relations to other areas in mathematics as well as to problems in theoretical physics \cite{Fomin_ICM, Ando, Karlin}. A more algebro-geometric approach to Lusztig's total positivity has been developed by Fock and Goncharov \cite{Fock_Goncharov}, and applied in the context of higher Teichm\"uller theory. 

In this article we describe a generalization of Lusztig's total positivity, which we call $\Theta$-positivity, and which is defined for other semisimple (resp. reductive) Lie groups which are not necessarily split. 
As a particular example $\Theta$-positivity includes the classical notion of positivity given by certain Lie semigroups in Lie groups of Hermitian type, which are related to bi-invariant orders and causality \cite{BenSimon_Hartnick, Hilgert_Neeb}. The notion of positivity in Lie groups of Hermitian type also played a role in recent developments of higher Teichm\"uller theory, notably in the theory of maximal representations, see \cite{Burger_Iozzi_Wienhard_toledo, BBHIW1, BBHIW2}. 

We classify semisimple Lie groups admitting a $\Theta$-positive structure, and show that besides split real Lie groups and Lie groups of Hermitian type, there are exactly two other families, namely groups locally isomorphic to $\SO(p,q)$, $p\neq q$, and the exceptional family of real forms of $F_4, E_6, E_7, E_8$, whose restricted root system is of type $F_4$. 
We describe several structure results for $\Theta$-positivity. We propose the notion of $\Theta$-positive representations of surface groups and conjecture that the spaces of $\Theta$-positive representations into the two new families of Lie groups admitting a $\Theta$-positive structure give new examples of higher Teichm\"uller spaces. 

A more detailed account to $\Theta$-positivity including all the proofs will appear in \cite{Guichard_Wienhard_pos}. The conjectures on $\Theta$-positive representations will be addressed in \cite{Guichard_Labourie_Wienhard}.

%
%
%
%
%
%
%
%
%
%
\section{Positivity in Lie groups}
In this section we shortly review several notions of positivity in Lie groups. 

\subsection{The positive reals and the order on the circle}\label{sec:sl2}
 Our starting point is the subset  $\RR_+\subset \RR$ of positive real numbers. 
Considering $\RR$ as a group, the subset $\RR_+$ is a sub-semigroup. Considering $\RR$ as a vector space, the subset $\RR_+$ is an open strict convex cone. 
Both viewpoints are important for more general notions of positivity. 

The cone $\RR_+$ is closely linked with the orientation on the circle $\RR\PP^1$. The tangent space of $\RR\PP^1$ to a point $x\in\RR\PP^1$ naturally identifies with $\RR$, and thus at every point $x\in \RR\PP^1$ the cone $\RR_+ \subset \RR \cong T_x\RR\PP^1$ provides a causal structure on $\RR\PP^1$. 

A triple of points $(x,y,z)$ on $\RR\PP^1$ is positively oriented if the points are pairwise distinct and read in this order going along the circle following its orientation. Positively oriented triples  in $\RR\PP^1$ can be described using the cone $\RR_+ \subset \RR$.
For simplicity we assume $ x = \RR e_2$ and $z = \RR e_1$, where $e_1, e_2$ are the standard basis vectors of $\RR^2$. The group $\mathrm{SL}(2,\RR)$ acts transitively on $\RR\PP^1$.  The subgroup 
$$
U =\Big\{ g \in \SL(2,\RR)\, |\, g = \begin{pmatrix} 1& t\\ 0& 1 \end{pmatrix}\Big\}
$$
fixes $z$ and acts transitively on $\RR\PP^1\backslash\{z\}$. 
The group $U$ is a one dimensional Abelian group and can be identified with $\RR$. An explicit identification is given by the map 
$$
\RR\longrightarrow U, \, t \mapsto \begin{pmatrix} 1& t\\ 0& 1 \end{pmatrix}
$$
The cone $\RR_+$ thus defines a subsemigroup $U_+ \subset U$, given by 
$$U_+ =\Big\{\begin{pmatrix} 1& t\\ 0& 1 \end{pmatrix} \in U |\, t>0\Big\}.$$

Any point $y\in \RR\PP^1 \backslash \{z\}$ can be written in a unique way as $u_y \cdot x$: 
If $y$ is the line spanned by a vector $t_y e_1 + e_2$, then $y = u_y \cdot x$ with $u_y = \begin{pmatrix} 1& t_y\\0& 1 \end{pmatrix}$. 
The triple $(x,y,z)$ is positively oriented if and only if $t_y>0$, i.e.\ if $u_y \in U_+$. 

We can go a step further and use the subsemigroup $U_+$ to define a subsemigroup of $\mathrm{SL}(2,\RR)$. 
For this we consider also the group 

$$
O=\Big\{ g \in \SL(2,\RR)\, |\, g = \begin{pmatrix} 1& 0\\ t& 1 \end{pmatrix}\Big\}
$$
and the subsemigroup 
$$
O_+ =\Big \{  \begin{pmatrix} 1& 0\\ t& 1 \end{pmatrix} \in O\, |\, t> 0\Big\}. 
$$
Let 
$$
A =\Big \{ g \in \SL(2,\RR)\, |\, g = \begin{pmatrix} \lambda& 0\\ 0& \lambda^{-1}\end{pmatrix}\Big\}, 
$$
be the subgroup of diagonal matrices, and 
$$
A^\circ = \Big\{ \begin{pmatrix} \lambda& 0\\ 0& \lambda^{-1}\end{pmatrix}\in A \, |\, \lambda>0\Big\} 
$$ 
which is the connected component of the identity in $A$. 
%
%
%

We define the subset $\SL(2,\RR)^{>0} \subset \SL(2,\RR)$ by 
$$
\SL(2,\RR)^{>0} =O_+ A^\circ U_+. 
$$
One can easily check that $\SL(2,\RR)^{>0}$ is the set of matrices all of whose entries are positive. From this it is immediate that $\SL(2,\RR)^{>0}$ is a subsemigroup of $\SL(2,\RR)$. 
The fact that $\SL(2,\RR)^{>0}$ is a subsemigroup can also be proved directly showing that the product of two elements in $O_+ A^\circ U_+$ is again in $O_+ A^\circ U_+$. 
The main point in this computation is to show that the product of an element of $U_+$ with and element of $O_+$ is again in $O_+ A^\circ U_+$, which can be checked explicitly:
$$
\begin{pmatrix} 1& s\\ 0& 1 \end{pmatrix} \begin{pmatrix} 1& 0\\ t& 1 \end{pmatrix}  = \begin{pmatrix} 1& 0\\ \frac{t}{1+st}& 1 \end{pmatrix} \begin{pmatrix} {1+st}& 0\\ 0& (1+st)^{-1} \end{pmatrix} 
\begin{pmatrix} 1& \frac{s}{1+st}\\ 0& 1 \end{pmatrix}.
$$
\subsection{Total positivity}
An $n\times n$-matrix is said to be totally positive if all of its minors are positive (i.e.\ in $\RR_+$). The set of all totally positive $n\times n$-matrices forms a subset 
$\mathrm{GL}(n,\RR)^{>0} \subset \mathrm{GL}(n,\RR)$. 
Totally positive matrices have very intriguing properties and have  many applications in various areas of mathematics ranging from statistics to representation theory, see for example \cite{Ando} for a survey. 

Totally positive matrices satisfy a decomposition theorem, and form in fact a semigroup. 
%
%
To describe this, let 
$U$ be the group of upper triangular matrices with ones on the diagonal, 
$O$ the group of lower triangular matrices with ones on the diagonal, and 
$A$ the group of diagonal matrices. 

Define $U^{>0} \subset U$, and $O^{>0}\subset O$ to be the subsets of totally positive unipotent matrices, i.e.\ those matrices of $U$, where all minors are positive, except those which have to be zero by the condition of being an element of $U$, similarly for $O$. 
Then $\mathrm{GL}(n,\RR)^{>0} $ satisfies the following decomposition theorem, due to A. Whitney \cite{Whitney}, see also \cite{Loewner}: 
$$\mathrm{GL}(n,\RR)^{>0} = O^{>0} A^\circ U^{>0}, $$
where $A^\circ$ is the connected component of the identity in $A$, i.e.\ diagonal matrices all of whose entries are positive.

The subsets $U^{>0}$ and $O^{>0}$ can be parametrized very explicitly.  The group $U$ is generated by elementary matrices 
$$
u_i(t) = I_n + t E_{i,i+1}, \, i = 1, \cdots, n-1, 
$$
where $I_n$ denotes the identity matrix and $E_{i, i+1}$ the matrix with the single entry $1$ in the $i$-th row and $(i+1)$-th column. 

The non-negative subsemigroup $U^{\geq 0}\subset U$ is the semigroup generated by all $u_i(t)$, $ i = 1, \cdots, n-1$ with $t \in \RR_+$. 
In the case when $n= 2$ this is already the  subsemigroup we are looking for, but for $n\geq 3$ the situation is a bit more complicated. 
In fact, the elements $u_i(t)  \in U^{\geq 0}$ are not contained in $U^{>0}$, since they have many minors which are zero, but should not be. 

In order to parametrize the set $U^{>0}$ we use the symmetric group $\mathcal{S}_{n}$ on $n$ letters.  We denote by $\sigma_i$, $ i = 1, \cdots n-1$,  the transposition $(i,i+1)$ and by $\omega_0$ be the longest element of the symmetric group, which sends $(1,2,\cdots , n-1, n)$ to  $(n, n-1, \cdots, 2, 1)$. For every way to write $\omega_0 =\sigma_{i_1} \sigma_{i_2} \cdots \sigma_{i_k}$, $k = \frac{n(n-1)}{2}$, as a reduced product of transpositions $\sigma_i$, $ i = 1, \cdots,  n-1$,
we define the map 
$$
F_{\sigma_{i_1} \sigma_{i_2} \cdots \sigma_{i_k}}: \RR^k \longrightarrow U, 
\, (t_1, \cdots t_k) \mapsto u_{i_1}(t_1)u_{i_2}(t_2) \cdots u_{i_k}(t_k). 
$$
An element is in $U^{>0}$ if and only if it is of the form $u_{i_1}(t_1)u_{i_2}(t_2) \cdots u_{i_k}(t_k)$ with $t_i \in \RR_+$ for all $ i = 1, \cdots, k$. 
The map $F_{\sigma_{i_1} \sigma_{i_2} \cdots \sigma_{i_k}}|_{(\RR_+)^k}$ is a bijection onto $U^{>0}$ and provides a parametrization of $U^{>0}$ by $(\RR_+)^k$. 
 
 There are many different ways to write $\omega_0$ as a reduced product of transpositions, and for two different reduced expression, the change of coordinates is given by a positive rational map. 
 
We illustrate this in the case when $n=3$. Here the longest element $\omega_0$ has two reduced expression 
$\sigma_1 \sigma_2 \sigma_1 = \omega_0= \sigma_2\sigma_1\sigma_2$. 
To compute the change of coordinates we consider  $F_{\sigma_1 \sigma_2 \sigma_1} (a,b,c) = u_{1}(a) u_2(b) u_1(c)$ and 
$F_{\sigma_2 \sigma_1 \sigma_2} (c',b',a') = u_{2}(c') u_1(b') u_2(a')$. 

$$
u_{1}(a) u_2(b) u_1(c)=  \begin{pmatrix} 1&a&0\\0&1&0\\0&0&1\end{pmatrix} \begin{pmatrix} 1&0&0\\0&1&b\\0&0&1\end{pmatrix}\begin{pmatrix} 1&c&0\\0&1&0\\0&0&1\end{pmatrix} = \begin{pmatrix} 1&a+c&ab\\0&1&b\\0&0&1\end{pmatrix},  
$$
and it is easy to see that $u_{1}(a) u_2(b) u_1(c) \in U^{>0}$ if and only if $a,b,c \in \RR_+$. 

$$
u_{2}(c') u_1(b') u_2(a')=  \begin{pmatrix} 1&0&0\\0&1&c'\\0&0&1\end{pmatrix} \begin{pmatrix} 1&b'&0\\0&1&0\\0&0&1\end{pmatrix}\begin{pmatrix} 1&0&0\\0&1&a'\\0&0&1\end{pmatrix} = \begin{pmatrix} 1&b'&b'a'\\0&1&c'+a'\\0&0&1\end{pmatrix}. 
$$
Comparing the matrix entries we get explicit transition maps $c' = \frac{bc}{a+c}$, $b' = a+c$, and $a' =\frac{ab}{a+c}$. 
If $a,b,c \in \RR_+$ then these maps are well defined, and $a',b',c' \in \RR_+$ as well. 


This explicit parametrization of $U^{>0}$ in fact follows Lusztig's approach, who generalized total positivity to arbitrary {\em split} real reductive Lie group $G$ \cite{Lusztig}. 
The reader who is familiar with the structure of reductive or semisimple Lie groups and their Lie algebras will see immediately that the one parameter subgroups $u_i(t)$  correspond to one parameter subgroups obtained when exponentiating the simple root spaces of the Lie algebra. And the role of the symmetric group $\mathcal{S}_n$ is in general played by the Weyl group of $G$. 
The proof that for a general split real reductive Lie group the changes of coordinates are given by postive rational maps reduces essentially to the above computation for the case when $n=3$.

%
%
%

\subsection{Convex cones and semigroups}
The subsemigroup $\SL(2,\RR)^{>0}\subset \SL(2,\RR)$ has been generalized in a different direction for semisimple Lie groups of Hermitian type, in particular for those of tube type as for example $\mathrm{Sp} (2n,\RR)$, $\mathrm{SU}(n,n)$, or $\mathrm{SO}(2,n)$. 
For this generalization we think of $\RR_+ \subset \RR$ as a strict convex cone with non-empty interior in the vector space $\RR$, which is homogeneous and invariant under the action of $\GL(1,\RR)$, defined by 
$$
\GL(1,\RR) \times \RR_+ \longrightarrow \RR_+ , \, (\lambda, v) \mapsto \lambda^2 v. 
$$

A Hermitian symmetric space $X = G/K$ is said to be of tube type if it is biholomorphically equivalent to a tube domain $T_\Omega = V + i \Omega$, where $V$ is a real vector space and $\Omega\subset V$ is an sharp convex cone. When $G = \SL(2,\RR)$ this tube domain is just the upper half space $\mathcal{H} = \RR + i\RR_+$. 
When $G = \mathrm{Sp}(2n,\RR)$ the tube domain is the Siegel upper half-space $\mathcal{H}_n =  \mathrm{Sym}(n,\RR) + i \mathrm{Pos}(n,\RR)$, where   $\mathrm{Pos}(n,\RR) \subset \mathrm{Sym}(n,\RR)$ is the subset of positive definite symmetric matrices. 

%
We focus on the example of  $G = \mathrm{Sp}(2n,\RR)$ and describe how the cone $\mathrm{Pos}(n,\RR) \subset \mathrm{Sym}(n,\RR)$ gives rise to a semigroup $\mathrm{Sp}(2n,\RR)^{\succ 0} \subset \mathrm{Sp}(2n,\RR)$. The construction of the semigroup $G^{\succ 0} \subset G$ for a general Lie group of Hermitian type and of tube type is analogous. 
For this we set
$$V = \Big\{g \in \mathrm{Sp}(2n,\RR) \, |\, g =  \begin{pmatrix} Id_n& 0\\ M& Id_n \end{pmatrix} \, ,\, M \in \mathrm{Sym}(n,\RR)\Big\},$$
$$W =\Big \{g \in \mathrm{Sp}(2n,\RR) \, |\, g =  \begin{pmatrix} Id_n& N\\ 0& Id_n \end{pmatrix} \,, \, N \in \mathrm{Sym}(n,\RR)\Big\},$$
and 
$$H = \Big\{g \in \mathrm{Sp}(2n,\RR) \, |\, g =  \begin{pmatrix} A & 0\\ 0& {A^t}^{-1} \end{pmatrix}\Big\} \cong \GL(n,\RR),$$
where the matrices are written with respect to a symplectic basis. 

The subsemigroup 
$\mathrm{Sp}(2n,\RR)^{\succ 0}$ is defined by  
$$ 
\mathrm{Sp} (2n,\RR)^{\succ 0} = V^{\succ 0} H^\circ W^{\succ 0},
$$
where 
$$V^{\succ 0} = \{ \begin{pmatrix} Id_n& 0\\ M& Id_n \end{pmatrix} \in V \, |\, M \in \mathrm{Pos}(n,\RR)\},$$
$$W^{\succ 0} = \{ \begin{pmatrix} Id_n& N\\ 0& Id_n \end{pmatrix} \in W \,|\,  N \in \mathrm{Pos}(n,\RR)\}, $$
and $H^\circ$ is the connected component of the identity of $H$. 
The fact that $\mathrm{Sp}(2n,\RR)^{\succ 0}$ is a subsemigroup follows by a computation similar to the computation given for $\SL(2,\RR)$ at the end of Section~\ref{sec:sl2}.

Since $\Sp(2n,\RR)$ is at the same time a split real Lie group as well as a Lie group of Hermitian type of tube type, we obtain two different subsemigroups,  the semigroup 
$\mathrm{Sp}(2n,\RR)^{>0}$ defined by Lusztig, and the semigroup $\mathrm{Sp}(2n,\RR)^{\succ 0}$. 
Note that $\Sp(2n,\RR)$ is the only simple Lie group which is at the same time a split real Lie group as well as a Lie group of Hermitian type of tube type. 

\section{Triple positivity in flag varieties}
Similarly to the relation between $\RR_+$ and positively oriented triples on $\RR\PP^1$, the two notions of positivity reviewed above lead to a notion of positivity of triples in certain flag varieties.

\subsection{Positivity in the full flag variety}\label{sec:triples}
The subsemigroup of totally positive matrices is closely related to the notion of positivity of triples in the space of full flags of vector subspaces in $\RR^n$. 
Let 
$$
\mathcal{F}:= \{ F=(F_1, F_2, \cdots, F_{n-1}) \, |\, F_i \subset \RR^n, \,  \dim(F_i) = i, \, F_i \subset F_{i+1}\} 
$$
denote the full flag variety. 
Two flags $F,F'$ are said to be transverse if $ F_i \cap F'_{n-i}  = \{ 0\} $. Given $F$ we denote by $\Omega_F$ the set of all flags in $\mathcal{F}$ which are transverse to $F$. $\Omega_F$ is an open and dense subset of $\mathcal{F}$.

We fix
$F \in \mathcal{F}$ to be the flag generated by the standard basis of $\RR^n$, i.e.\  $F_i = \mathrm{span} (e_1, \cdots, e_i)$, and 
$E \in \mathcal{F}$ to be the flag generated by the standard basis of $\RR^n$ in the opposite order, i.e.\  $E_i = \mathrm{span} (e_n, \cdots, e_{n-i+1})$. 

Any flag $T$ which is transverse to $F$, is the image of $E$ under a unique element $u_T \in U$. 
The triple of flags $(E, T, F)$ is said to be {\em positive} if and only if $T  = u_T \cdot E$ for an element $u_T \in U^{>0}$. 

If $(E,T,F)$ is positive, then $T$ is automatically  transverse to $E$. In fact, Lusztig \cite{Lusztig} proved that the set $\{ T \in \mathcal{F}\,|\,  (E,T,F) \text{ is positive }\}$ is a connected component of the intersection $\Omega_E \cap \Omega_F$. This connected component, which can be identified with $U^{>0}$, carries the structure of a semigroup. 

Since $\GL(n,\RR)$ acts transitively on pairs of tranvserse flags, any two transverse flags $(F_1, F_2)$ can be mapped to $(E,F)$ by an element of $\GL(n,\RR)$ and we can extend the notion of positivity to any triple of flags. 

\begin{remark}
Note that by this extension through the $\GL(n,\RR)$ action, one might loose some of the geometric properties. For example, when $n=2$, any triple of pairwise distinct points in $\RR\PP^1$ is positive, since there is an element in $\GL(2,\RR)$ which changes the orientation of $\RR\PP^1$. If we extend the notion of positivity of triple via the $\SL(2,\RR)$ action we obtain the positively oriented triples discussed above. In general it turns out to be useful to work not only with positive triples, but with positive four-tuple or more generally postiive $n$-tuples, see Remark~\ref{rem:fourtuples}.  The set of positive triples (and more generally $n$-tuples) of flags admits a very explicit description in terms of projective invariants (triple ratios and crossratios) of flags, see \cite{Fock_Goncharov}. 
\end{remark}

In a very analogous way, Lusztig's total positivity in a split real semisimple Lie group $G$ is linked to notion of positivity of triples in the generalized flag variety $G/B$, where $B$ is the Borel subgroup of $G$.

\subsection{Positivity and the Maslov index}\label{sec:maslov}
If $G$ is a Lie group of Hermitian type, which is of tube type, the positive structure $G^{\succ 0} \subset G$ is also linked with the notion of positivity of triple in the generalized partial flag variety $G/Q$, which arises as the Shilov boundary of the Hermitian symmetric space. We illustrate this for $G = \Sp(2n,\RR)$ where the relevant partial flag variety is the space of Lagrangian subspaces. 

Let $\omega$ be the standard symplectic form on $\RR^{2n}$ and let $\{ e_1, \cdots, e_n, f_1, \cdots f_n\}$ be a symplectic basis of $\RR^{2n}$ with respect to $\omega$. 

Let 
$$
\mathcal{L} := \{L \subset \RR^{2n} \, |\, \dim{L} = n, \, \omega|_{L \times L} = 0 \}
$$
be the space of Lagrangian subspaces. 
Two Lagrangians $L$ and $L'$ are transverse if $L \cap L' = \{ 0\}$. We denote by $\Omega_L$ the set of Lagrangians transverse to $L$. 

Fix 
$L_E = \mathrm{span} (e_1, \cdots, e_n)$ and $L_F = \mathrm{span} (f_1, \cdots, f_n)$. 
Any Lagrangian $L_T \in \mathcal{L}$ transverse to $L_F$ is the image of $L_E$ under an element $v_T =\begin{pmatrix} Id_n& 0\\ M_T& Id_n \end{pmatrix}  \in V$. 

The triple of Lagrangians $(L_E, L_T, L_F)$ is said to be {\em  positive} if and only if $M_T \in \mathrm{Pos}(n,\RR) \subset \mathrm{Sym}(n,\RR)$.  

%

Again the set of $L_T \in \mathcal{L}$ such that $(L_E, L_T, L_F)$ is positive is a connected component of $\Omega_{L_E} \cap \Omega_{L_F}$.  
 
The symplectic group $ \Sp(2n,\RR)$ acts transitively on $\mathcal{L}$ and on the space of pairs of transverse Lagrangians. 
The stabilizer of the two Lagrangian subspaces $L_E$ and $L_F$  is 
$ \mathrm{Stab}_{\Sp(2n,\RR)} (L_E) \cap  \mathrm{Stab}_{\Sp(2n,\RR)}(L_F) = H \cong \GL(n,\RR) $.
 
If $h = \begin{pmatrix} A & 0\\ 0& {A^t}^{-1}\end{pmatrix}$ and  $v_T =\begin{pmatrix} Id_n& 0\\ M_T& Id_n \end{pmatrix}$ is the element with $v_T\cdot L_E = L_T$. Let $L_{T'} = h \cdot v_T \cdot L_E$, then 
$L_{T'} = v_{T'} \cdot L_E$ with $v_{T'} = \begin{pmatrix} Id_n& 0\\ h M_T h^t & Id_n \end{pmatrix}$. Since $h M_T h^t$ is positive definite if and only if $M_T$ is, we can extend notion of positivity to all triples of Lagrangians using the action of $\Sp(2n,\RR)$. 
 
The notion of positive triples in $\mathcal{L}$ is in fact closely related to the Maslov index $\mu: \mathcal{L}^3 \to \ZZ$. A triple of Lagrangians $(L_1, L_2, L_3)$ is positive if and only if the $\mu(L_1, L_2, L_3) = n$. In an analogous way, the notion of positivity of a triple of points in the Shilov boundary of a Hermitian symmetric space of tube type can be defined. Positive triples are then characterized by the generalized Maslov index \cite{Clerc_Orsted, Clerc} assuming its maximal value. 
 
 \begin{remark}
 Note that the tangent space at the point $L_E$ naturally identifies with $\mathrm{Sym}(n,\RR)$ and the convex cone $\mathrm{Pos}(n,\RR) \subset \mathrm{Sym}(n,\RR)$ defines a causal structure on $\mathcal{L}$, see \cite{Kaneyuki} for more details. 
 \end{remark}

\section{$\Theta$-positivity}
In the previous section we reviewed classical notions of positivity in semisimple Lie groups and in associated flag varieties and described them in such a way as to underline their similarities. 
In this section we will show that these notions of positivity in  split real forms and Lie groups of Hermitian type are particular cases of a more general notion of positivity, which we call $\Theta$-positivity, where $\Theta \subset \Delta$ is a subset of simple positive roots. The classification of simple Lie groups admitting a $\Theta$-positive structure includes two more families, the groups $\mathrm{SO}(p,q)$, $p\neq q$, and an exceptional family. 

In order to describe $\Theta$-positivity we will recall some facts about semisimple Lie groups $G$ and the structure of the Lie algebras of parabolic groups $P_\Theta <G$ defined by $\Theta\subset \Delta$. The definition of $\Theta$-positive structures will be given in terms of properties of these Lie algebras. However in Theorem~\ref{cor:positive} we deduce a more geometric  characterization in terms of the structure of triples of points in the flag variety $G/P_\Theta$.  We refer the reader not familiar with semisimple Lie algebras to \cite{Helgason} for more background. 
In Section~\ref{sec:so3q} we give a more elementary description in the case when $G = \SO(3,q)$, $ q> 3$, which does not require any knowledge about the structure theory of semisimple Lie groups. 

\subsection{Structure of parabolic subgroups} 
In order to set notation, let  $G$ be a real semisimple Lie group (with finite center),  $\mathfrak{g}$ its Lie algebra, and denote by $\mathfrak{k}$ the Lie algebra of a maximal compact subgroup $K<G$. Then $\mathfrak{g} = \mathfrak{k} \oplus \mathfrak{k}^\perp$ where $\mathfrak{k}^\perp$ is the orthogonal complement with respect to the Killing form on $\mathfrak{g}$. We choose $\mathfrak{a} \subset \mathfrak{g}$ a maximal Abelian Cartan subspace in $\mathfrak{k}^\perp$, and denote by  $\Sigma = \Sigma(\mathfrak{g}, \mathfrak{a})$ the system of restricted roots. We choose $\Delta \subset \Sigma$ a system of simple roots, and let  $\Sigma^+$ denote the set of positive roots, and $\Sigma^-$ the set of negative roots.
Let $\Theta \subset \Delta$ be a subset. We set 
$$\mathfrak{u}_\Theta = \sum_{\alpha \in \Sigma_\Theta^+} \mathfrak{g}_{\alpha},\,\mathfrak{u}^{opp}_\Theta = \sum_{\alpha \in \Sigma_\Theta^+} \mathfrak{g}_{-\alpha} $$ 
where  $\Sigma_\Theta^+ = \Sigma^+ \backslash (\mathrm{Span}(\Delta - \Theta))$, and 
$$\mathfrak{l}_\Theta =\mathfrak{g}_0 \oplus  \sum_{\alpha \in \mathrm{Span}(\Delta - \Theta)\cap \Sigma^+} (\mathfrak{g}_{\alpha} \oplus \mathfrak{g}_{-\alpha}).$$

Then the standard parabolic subgroup $P_\Theta$ associated to $\Theta \subset \Delta$ is the normalizer in $G$ of $\mathfrak{u}_\Theta$. 
We also denote by $P^{opp}_\Theta$ the normalizer in $G$ of $\mathfrak{u}^{opp}_{\Theta}$.

The group $P_\Theta$ is the semidirect product of its unipotent radical $U_\Theta: = \exp(\mathfrak{u}_\Theta)$ and the Levi subgroup $L_\Theta = P_\Theta \cap P^{opp}_\Theta$. 
The Lie algebra of $L_\Theta$ is $\mathfrak{l}_\Theta$. In particular the Lie algebra $\mathfrak{p}_\Theta$ of $P_\Theta$ decomposes as 
$ \mathfrak{p}_\Theta = \mathfrak{l}_\Theta \oplus \mathfrak{u}_\Theta$. 
Note that with our convention  $P_{\emptyset} = G$ and $P_{\Delta}$ is the minimal parabolic subgroup. 

The Levi subgroup $L_\Theta$ acts via the adjoint action on $\mathfrak{u}_\Theta$. 
Let $\mathfrak{z}_\Theta$ denote the center of $\mathfrak{l}_\Theta$. 
Then $\mathfrak{u}_\Theta$ decomposes into the weight spaces $\mathfrak{u}_\beta$, $\beta \in \mathfrak{z}^*_\Theta$, 
$$
\mathfrak{u}_\beta: = \{ N \in \mathfrak{u}_\Theta \, |\, \mathrm{ad}(Z)N = \beta(Z) N , \, \forall Z \in \mathfrak{z}_\Theta \}. 
$$

Note that 
$$
\mathfrak{u}_\beta = \sum_{\alpha \in \Sigma^+_\Theta, \, \alpha|_{\mathfrak{z}_\Theta} = \beta }  \, \, \mathfrak{g}_{\alpha} .
$$
There is a unique way to write $\beta \in \mathfrak{z}^*_\Theta$ as the restriction of a root in $\mathfrak{a}^*$ which lies in the span of $\Theta$ to $\mathfrak{z}_\Theta$, so, with a slight abuse of notation we consider $\beta$ as an element of $\mathfrak{a}^*$ and write: 
$$ \mathfrak{u}_\beta =\sum_{\alpha \in \Sigma^+_\Theta, \alpha= \beta \, mod\,  \mathrm{Span}(\Delta - \Theta)}\, \,  \mathfrak{g}_{\alpha}.$$

Any $\mathfrak{u}_\beta$ is invariant by $L_\Theta$ and is an irreducible representation of $L_\Theta$. 
The relation $[\mathfrak{u}_\beta,\mathfrak{u}_{\beta'}] \subset \mathfrak{u}_{\beta + \beta'}$ is satisfied, and the Lie algebra $\mathfrak{u}_\Theta$ is generated by the $\mathfrak{u}_\beta$ with $\beta \in \Theta$. We call $\mathfrak{u}_\beta$ with $\beta \in \Theta$ indecomposable. 

\begin{examples} 
\begin{enumerate}
\item Let $G$ be a split real form, and $\Theta = \Delta$. Then $\mathfrak{u}_\beta = \mathfrak{g}_\beta$ for all $\beta \in \Sigma^+$, and $\mathfrak{u}_\beta$ is indecomposable if $\beta \in \Delta$. 
\item Let $G$ be a Lie group of Hermitian type. Then the root system is of type $C_r$ if $G$ is of tube type, and of type $BC_r$ (non-reduced) if $G$ is not of tube type. 
Let $\Delta =\{ \alpha_1, \cdots  , \alpha_r\}$, and let $\Theta  = \{ \alpha_r\}$ be the subset such that $P_\Theta$ is the stabilizer of a point in the Shilov boundary of the Hermitian symmetric space associated to $G$. 
Then $\mathfrak{u}_\Theta = \mathfrak{u}_{\alpha_r}$ if $G$ is of tube type, and $\mathfrak{u}_\Theta= \mathfrak{u}_{\alpha_r} \oplus \mathfrak{u}_{2 \alpha_r}$ if $G$ is not of tube type. 
\end{enumerate}
\end{examples}

\subsection{$\Theta$-positive structures} 
We can now give the definition of $\Theta$-positive structures. 
\begin{definition} 
Let $G$ be a semisimple Lie group with finite center. Let $\Theta \subset \Delta$ be a subset of simple roots. 
We say that $G$ admits a $\Theta$-positive structure if for all $\beta \in \Theta$ there exists an $L^\circ_\Theta$-invariant sharp convex cone in $\mathfrak{u}_\beta$.
\end{definition}

For a more geometric characterization of $\Theta$-positivity we refer to Theorem~\ref{cor:positive} below.

\begin{theorem}\label{thm:classification}
A semisimple Lie group $G$ admits a $\Theta$-positive structure if and only if $(G, \Theta)$ belongs to one of the following four cases: 
\begin{enumerate}
\item $G$ is a split real form, and $\Theta = \Delta$. 
\item $G$ is of Hermitian type of tube type and $\Theta = \{ \alpha_r\}$. 
\item $G$ is locally isomorphic to $\SO(p,q)$, $p \neq q$, and $\Theta = \{ \alpha_1, \cdots, \alpha_{p-1}\}$.
\item $G$ is a real form of $F_4, E_6, E_7, E_8$, whose restricted root system is of type $F_4$, and $\Theta = \{ \alpha_1, \alpha_2\}$. 
\end{enumerate}
\end{theorem}

\begin{proof}
We make use of necessary and sufficient conditions for the existence of an invariant sharp convex cone in the vector space $V$ of an irreducible representation $H \to \GL(V)$ of a connected reductive group $H$, see for example \cite[Proposition~4.7]{Benoist}.
Such an invariant cone exists if and only if the representation is proximal, i.e.\ the highest weight space is one dimensional, and  if the highest weight is contained in $2P$, where $P$ is the weight lattice. Another equivalent definition is that the maximal compact subgroup of $H$  has a nonzero invariant vector in $V$.

Applying this to the case when $H = L^\circ_\Theta$ and $V = \mathfrak{u}_\beta$ this leads to the following necessary and sufficient criteria for the Dynkin diagram of the system of restricted roots $\Sigma$: 


\begin{enumerate}
\item $ \forall \beta \in \Theta$ the root space $\mathfrak{g}_\beta$ is one-dimensional. (\emph{representation is proximal}) 
\item $\forall \beta \in \Theta$ the node of the Dynkin diagram with label $\beta$ is either connected to the nodes in $\Delta-\Theta$ by a double arrow pointing towards $\Delta-\Theta$, or it is not connected to $\Delta-\Theta$ at all. (\emph{highest weight is in $2P$})
\end{enumerate}
From this we deduce the above list of pairs $(G,\Theta)$.
\end{proof}

In order to describe the $\Theta$-positive structure in more detail we denote for every $\beta \in \Theta$ by $c_\beta \subset \mathfrak{u}_\beta$ the $L^\circ_\Theta$ invariant closed convex cone, by $c^\circ_\beta$ its interior, and by $C_\beta: = \mathrm{exp}(c_\beta) \subset U_\beta = \mathrm{exp}(\mathfrak{u}_\beta)$ its image in $U_\beta \subset U_\Theta$. 

Note that by the classification of Dynkin diagrams there is at most one node in  $ \Theta$ which is connected to $\Delta - \Theta$, we denote this node/simple root by $\beta_\Theta$. 
For all $\beta \in \Theta\backslash \{ \beta_{\Theta}\}$ we then have that $\mathfrak{u}_\beta \cong \RR$, and $c_\beta \cong \RR_+$. For $\beta_\Theta$, the vector space $\mathfrak{u}_{\beta_\Theta}$ is of dimension $\geq 2$, and $c_{\beta_\Theta} \subset \mathfrak{u}_{\beta_\Theta}$ is a homogeneous sharp convex cone.

We define the $\Theta$-nonnegative semigroup $U_\Theta^{\geq 0}$ to be the subsemigroup of $U_\Theta$ generated by $C_\beta$, $\beta \in \Theta$. 
Similarly we consider the cones $c^{opp}_\beta$, $c^{opp, \circ}_\beta$ and $C^{opp}_\beta = \mathrm{exp}(c^{opp}_\beta)$, and denote by ${U^{opp}_\Theta}^{\geq 0}$ the subsemigroup of $U^{opp}_\Theta$ generated by  $C^{opp}_\beta$, $\beta \in \Theta$.

We define the $\Theta$-nonnegative semigroup $G^{\geq 0}_\Theta$ to be the subsemigroup generated by $U_\Theta^{\geq 0}$, ${U^{opp}_\Theta}^{\geq 0}$, and $L_{\Theta,\circ}$, where $L_{\Theta,\circ}$ is the connected component of the identity in $L_\Theta$. 

%

\begin{examples} 
\begin{enumerate}
\item When $G$ is a split real form, and $\Theta = \Delta$. Then $G^{\geq 0}_\Theta = G^{\geq 0 }$ is the set of nonnegative elements, defined by Lusztig. 

\item Let $G$ be a Lie group of Hermitian type, which is of tube type and $\Theta  = \{ \alpha_r\}$. 
Then  $G^{\geq 0}_\Theta$ is the closure of the subsemigroup $G^{\succ 0} \subset G$ defined above. 
\end{enumerate}
\end{examples}

\subsection{The $\Theta$-positive semigroup}
In this section we define the $\Theta$-positive semigroups $U^{>0}_\Theta \subset U_\Theta$ and $G^{>0}_\Theta \subset G$, and give an explicit parametrization of $U^{>0}_\Theta $. 

Let $G$ be a semisimple Lie group with a $\Theta$-positive structure. 
We associate to $\Theta$ a subgroup $W(\Theta)$ of the Weyl group $W$. Recall that the Weyl group $W$ is generated 
by the reflections $s_\alpha$, $\alpha \in \Delta$. 
We set $\sigma_\beta = s_\beta $ for all $\beta \in \Theta - \{ \beta_\Theta\}$, and define $\sigma_{\beta_\Theta}$ to be the longest element of the Weyl group $W_{\{\beta_\Theta\} \cup (\Delta - \Theta)}$ of the subrootsystem generated by $\{\beta_\Theta\} \cup \Delta - \Theta$, i.e.\ $W_{\{\beta_\Theta\} \cup \Delta - \Theta} \subset W$ is the subgroup generated by $s_{\alpha}$ with $\alpha \in \{\beta_\Theta\} \cup \Delta - \Theta$. 

The group $W(\Theta)\subset W$ is defined to be the subgroup generated by $\sigma_\beta$, $\beta \in \Theta$. 

It turns out that the group $W(\Theta)$ is isomorphic to a Weyl group of a simple root system, in such a way that the $\sigma_\beta$ correspond to standard generators. If we denote the Weyl group of a root system $\Sigma'$ by $W_{\Sigma'}$ we have: 
\begin{enumerate}
\item If $G$ is a split real form and  $\Theta = \Delta$, $W(\Theta) = W$.
\item If $G$ is of Hermitian type of tube type and $\Theta = \{ \alpha_r\}$, $W(\Theta) \cong W_{A_1}$. 
\item If $G$ is locally isomorphic to $\SO(p,q)$, $p \neq q$, and $\Theta = \{ \alpha_1, \cdots, \alpha_{p-1}\}$, $W(\Theta) \cong W_{B_{p-1}}$.
\item If $G$ is a real form of $F_4, E_6, E_7, E_8$, whose restricted root system is of type $F_4$, and $\Theta = \{ \alpha_1, \alpha_2\}$, $W(\Theta) \cong W_{G_2}$. 
\end{enumerate}

The group $W(\Theta)$ acts on the weight spaces $\mathfrak{u}_\beta$, $\beta \in \mathrm{span}(\Theta)$. 
Combinatorically this action is the same as the action of the Weyl group $W_{\Sigma'}$  with $W(\Theta) \cong W_{\Sigma'}$ on the root spaces $\mathfrak{g'}_\alpha$, $\alpha \in \Sigma'$. 

For any $\beta \in \Theta$ we define a map 
$$
x_\beta: \mathfrak{u}_\beta \longrightarrow U_\beta \subset U_\Theta, \, \, v \mapsto \exp(v). 
$$

Let $w_\Theta^0 \in W(\Theta)$ be the longest element, and let $w_\Theta^0 = \sigma_{i_1} \cdots \sigma_{i_l}$ be a reduced expression. 
We define a map 
$$
F_{\sigma_{i_1} \cdots \sigma_{i_l}} : c^\circ_{\beta_{i_1}}\times \cdots \times c^\circ_{\beta_{i_l} }\longrightarrow U_\Theta
$$

by
$$
 (v_{i_1}, \cdots, v_{i_l}) \mapsto x_{\beta_{i_1}}(v_{i_1})\cdots x_{\beta_{i_l}}(v_{i_l})
 $$
 
 \begin{theorem}
 The image $U_\Theta^{>0}:= F_{\sigma_{i_1} \cdots \sigma_{i_l}} ( c^\circ_{\beta_{i_1}}\times \cdots \times c^\circ_{\beta_{i_l} }) $ is independent of the reduced expression of $w_\Theta^0 $. 
We call $U_\Theta^{>0}$ the $\Theta$-positive semigroup of $U_\Theta$. 
 \end{theorem}

We just sketch the idea of the proof, which is inspired by the strategy of Berenstein and Zelevinsky \cite{Berenstein_Zelevinsky} to compute the transition functions for total positivity in split real Lie groups. The whole proof is rather involved and will appear in \cite{Guichard_Wienhard_pos}. The explicit formulas for the case when $G = \SO(3,q)$, $q>3$, are given in Section~\ref{sec:so3q}.
\begin{proof}
Any two reduced expression of $w_\Theta^0 $ differ by a braid relation in the Weyl group $W(\Theta)$. There are three possible braid relations: $\sigma_i \sigma_j \sigma_i = \sigma_j \sigma_i \sigma_j$, $\sigma_i \sigma_j \sigma_i \sigma_j= \sigma_j \sigma_i \sigma_j\sigma_i$, and $\sigma_i \sigma_j \sigma_i \sigma_j \sigma_i \sigma_j = \sigma_j \sigma_i \sigma_j \sigma_i \sigma_j \sigma_i $.
Using explicit calculations in the universal enveloping algebra of $\mathfrak{u}_\Theta$ we get explicit systems of ``polynomial" equations for each braid relation, which combinatorically have the same structure as the polynomial equations described in \cite{Berenstein_Zelevinsky} for the totally positive semigroup in the split real group whose Weyl group $W_{\Sigma'}$ is isomorphic to $W(\Theta)$, if the equations are interpreted in the right way,  because now some of the variables are {\em not} scalars, but vectors. 

For the first kind of braid relation, the involved cones $c_{\beta_i}$ and $c_{\beta_j}$ are always $\RR_+$, and the computation reduces basically to the computation in the $\SL(3,\RR)$-case discussed above.
 
For the second kind, which appears for example when $G = \SO(p,q)$, one of the cones, say $c_{\beta_j}$ is the cone of positive vectors in a vector space, which is equipped with a quadratic form $B$ of signature $(1, q-p+1)$, whose first component is positive. In this case even powers in the polynomial equation need to be interpreted as $v^{2k} = B(v,v)^k$, and odd powers as $v^{2k+1} = B(v,v)^k v$.
The computations reduce essentially to the case when $(p,q) = (3,q)$, for which we give the precise formulas in Section~\ref{sec:so3q}.

For the third kind of braid relation, which appear when $G$ is a real form of $F_4, E_6, E_7$, or $E_8$ and $\Theta = \{ \alpha_1, \alpha_2\}$, the cone $c_{\beta_j}$ is identified with the cone of positive definite matrices in the vector space of Hermitian $3\times 3$ matrices over $\RR, \CC, \HH$, or $\OO$. Here the square of $v$ needs to be suitably interpreted as the Freudenthal product $v\times v$, and the third power $v^3$ or terms of the form $v^2w$ as the trilinear form $(v,v,w)$, see e.g.\ \cite{Yokota} for the definition of the Freudenthal product and the trilinear form. Here the determination of the ``polynomial" equations and their solution is much more complicated and will be given in \cite{Guichard_Wienhard_pos}.

Once the explicit system of ``polynomial" equations is determined, we prove that starting with variables $v_{i}, v_j$ in the cones $c^\circ_{\beta_{k}}$, $k = i,j$, there exists a unique solution, which gives us the transition functions. These transition functions are then shown to take again values in the open cones $c^\circ_{\beta_{l}}$, $l = j,i$. As a consequence the semigroup $U_\Theta^{>0}$ is well defined. 
\end{proof}

We define the $\Theta$-positive semigroup $G^{> 0}_\Theta$ to be the subsemigroup generated by $U_\Theta^{>0}$, ${U^{opp}_\Theta}^{> 0}$, and $L_{\Theta}^\circ$. 

\subsection{$\Theta$-positivity of triples} 
Let $G$ be a semisimple Lie group with a $\Theta$-positive structure. 
We consider the generalized flag variety $G/P_\Theta$. Given $F \in G/P_\Theta$ we denote by $\Omega_{F_\Theta}$ the set of all points in $G/P_\Theta$ which are transverse to $F$. 

We fix $E_\Theta$ and $F_\Theta$ to be the standard flags 
such that $\mathrm{Stab}_G(F_\Theta) = P_\Theta$ and $\mathrm{Stab}_G(E_\Theta) = P^{opp}_\Theta$. 
Given any $S_\Theta \in G/P_\Theta$ transverse to $F_\Theta$, there exists $u_{S_\Theta} \subset U_\Theta$ such that $S_\Theta = u_{S_\Theta} E_\Theta$. 
\begin{definition}
The triple $(E_\Theta, S_\Theta, F_\Theta)$ is $\Theta$-positive if $u_{S_\Theta} \in U^{>0}_{\Theta}$. 
\end{definition}

\begin{theorem}
The set 
$$
\{S_\Theta \in G/P_\Theta\, |\, (E_\Theta, S_\Theta, F_\Theta)\text{ is } \Theta\text{-positive}\} 
$$
is a connected component of the intersection  $\Omega_{E_\Theta}\cap \Omega_{F_\Theta}$. This connected component has the structure of a semigroup. 
\end{theorem}

This property can in fact be proven to give a more geometric characterization of a $\Theta$-positive structure on $G$. 

\begin{theorem}\label{cor:positive}
$G$ has a $\Theta$-positive structure if and only if there are two transverse points $E_\Theta, F_\Theta \in G/P_\Theta$ such that a connected component of $\Omega_{E_\Theta}\cap \Omega_{F_\Theta}\subset  G/P_\Theta$ has the structure of a semigroup. 
\end{theorem}

\begin{remark}\label{rem:fourtuples}
Note that this connected component is unique up to exchanging cones $c^\circ_\beta$ with $-c^\circ_\beta$. 

We can define not only the notion of positive triples, but more generally the notion of positive $n$-tuples in $G/P_\Theta$. For example, a four tuple $(E_\Theta, S_\Theta, S'_\Theta, F_\Theta)$ is positive if and only if the two triples $(E_\Theta, S_\Theta, F_\Theta)$ and $(S_\Theta, S'_\Theta, F_\Theta)$ are positive, and the connected component of  $\Omega_{S_\Theta}\cap \Omega_{F_\Theta}$ containing $S'_\Theta$ is contained in the connected component of $\Omega_{E_\Theta}\cap \Omega_{F_\Theta}$ containing $S_\Theta$. 
\end{remark}

We expect that many of the properties which Lusztig proved for totally positivity in split real Lie groups, have an appropriate analogue in the context of $\Theta$-positive, even though Lusztig's proofs sometimes make use of the positivity of the canonical basis for the quantum universal enveloping algebra, for which we do not (yet?) have an analogue here. 

For example we conjecture 
\begin{conjecture}
Any $g \in G_\Theta^{>0}$ acts proximally on $G/P_\Theta$, i.e $g$ has a unique attracting and a unique repelling fix point in $G/P_\Theta$, which are transverse, and such that the action of $g$ on the tangent space at the attracting fix point is strongly contracting, and the action on the tangent space at the repelling fix point is strongly expanding. 
\end{conjecture}

\begin{conjecture}\label{conj:nilpotent}
Let $v = \sum_{\beta \in \Theta} v_\beta$, with $v_\beta \in c^\circ_\beta$. Then $\exp(v) \in U_\Theta^{>0}$. Conversely, if $v \in \mathfrak{u}_\Theta$ with $\exp(tv) \in U_\Theta^{>0}$, then $v = \sum_{\beta \in \Theta} v_\beta$ with $v_\beta \in c^\circ_\beta$. 
\end{conjecture}

From  the second conjecture  we can deduce the following statement. 

\begin{conjecture}\label{conj:sl2}
Let $\mathfrak{g}_0  = \langle e, f, h\rangle  \subset  \mathfrak{g}$ be a three dimensional simple Lie algebra where $e$ denotes the nilpotent element.  
If  $e = \sum_{\beta\in \Theta} v_\beta$ with $ v_\beta \in c_\beta^\circ$ for all $\beta \in \Theta$, then the totally geodesic embedding $\HH^2 \to G/K$ extends to a continuous equivariant  map $\RR\PP^1 = \partial \HH^2 \to  G/P_\Theta$, which sends positive triples in $\RR\PP^1$ to positive triples in $G/P_\Theta$. 
\end{conjecture}

In the case of Lusztig's total positivity these properties are proven in \cite{Lusztig}. For Hermitian Lie groups of tube type the properties can be checked directly and Conjecture~\ref{conj:sl2} can be deduced from \cite{Burger_Iozzi_Wienhard_tight}. 

\subsection{The positive structure for $\SO(3,q)$}\label{sec:so3q}
Here we describe in detail the $\Theta$-positive structure for the group $SO(3,q)$, $q>3$. The general case of $SO(p,q)$, $p\neq q$ essentially reduces to the computations in this subsection and the computations for $\SL(3,\RR)$. 

Let $Q$ be a quadratic form of signature $(3,q)$. We write $Q = \begin{pmatrix} 0& 0& K\\ 0 &J &0 \\ -K&0& 0 \end{pmatrix}$, 
where $K= \begin{pmatrix} 0& 1\\ -1&0 \end{pmatrix}$, and $J = \begin{pmatrix} 0&0& 1\\ 0& -Id_{q-3} &0\\ 1 & 0& 0 \end{pmatrix}$, so that $Q(a,b) = a^t Q b$. 
We set $b_J  (v,w):= \frac{1}{2} v^t J w$, $q_J (v) = b_J(v,v)$.

Consider $G = \SO(Q) = SO(3,q)$. 
We choose the Cartan subspace $\mathfrak{a}$ to be the set of diagonal matrices $diag(\lambda_1, \lambda_2, \lambda_3, 0, \cdots,0, -\lambda_3, -\lambda_2, -\lambda_1)$. 
We denote by $\epsilon _i: \mathfrak{a} \to \RR$ the linear form which sends such a diagonal matrix to $\lambda_i$. 
We choose a set of simple roots $\Delta= \{\alpha_1, \alpha_2, \alpha_3\}$ with $\alpha_i = \epsilon_i - \epsilon_{i+1}$, $i = 1,2$, $\alpha_3 = \epsilon_3$, and take $\Theta = \{\alpha_1, \alpha_2\}$. 

The generalized flag variety $G/P_\Theta$ is 
$$
\mathcal{F}_{1,2} =\Big \{ V_1 \subset V_2 \, |\, \dim(V_i) = i, Q|_{V_2\times V_2} = 0\Big\} 
$$
We consider $F = (F_1, F_2)$ with $F_1 = \RR e_1$ and $F_2 = F_1 \oplus \RR e_2$, and $E= (E_1, E_2)$ with $E_1  = \RR e_{q+3}$, and $E_2 = E_1 \oplus\RR e_{q+2}$. 

Then $P_\Theta$ is the stabilizer of $F$, and its unipotent subgroup is 
$$
U_\Theta = \{ U(x,v,w,a) \, |\, a,x \in \RR, v, w \in \RR^{q-1}\}, $$
where
$$
U(x,v,w,a) = 
\begin{pmatrix} 
1& x& w^t+ x \frac{v^t}{2}& a & au - q_J(w+\frac{v}{2})\\
0& 1& v^t& q_J(v)& a - 2 b_J(v,w)\\
0&0&Id_{q-1}& J v &-J w+ x J \frac{v}{2}\\
0&0&0&1&x\\
0&0&0&0&1
\end{pmatrix}
$$

Here $\mathfrak{u}_{\alpha_1}$ is equal to $\RR$ and $\mathfrak{u}_{\alpha_2}$ is equal to  $\RR^{q-1}$, endowed with the form $q_J$. 
The maps $x_{\alpha_1}:  \RR \to U_\Theta$ and $x_{\alpha_2} :\RR^{q-1} \to U_\Theta$ are given by 
$$
x_{\alpha_1}(x) = \exp \begin{pmatrix} 
0& x& 0& 0 & 0\\
0& 0& 0& 0&0\\
0&0&0& 0&0\\
0&0&0&0&x\\
0&0&0&0&0
\end{pmatrix} = \begin{pmatrix} 
1& x& 0& 0& 0\\
0& 1& 0& 0&0\\
0&0&Id_{q-1}&0& 0\\
0&0&0&1&x\\
0&0&0&0&1
\end{pmatrix}
$$

$$
x_{\alpha_2}(v) = \exp \begin{pmatrix} 
0& 0& 0& 0 & 0\\
0& 0& v^t& 0&0\\
0&0&0& J v&0\\
0&0&0&0&0\\
0&0&0&0&0
\end{pmatrix} = \begin{pmatrix} 
1& 0& 0& 0& 0\\
0& 1& v^t& q_J(v)&0\\
0&0&Id_{q-1}&J v& 0\\
0&0&0&1&0\\
0&0&0&0&1
\end{pmatrix}
$$

The cone $c^\circ_{\alpha_1}$ is equal to $\RR_+$, and the cone $ c^\circ_{\alpha_2}$ is equal to $\{ v\in \RR^{q-1} \, |\, q_J(v) >0, \mathrm{sign}(v_1)>0\}$.

The group $W(\Theta)$ is isomorphic to the Weyl group $W_{B_2}$ of the root sytem $B_2$. 
Let $\omega^0_{\Theta}=\sigma_1\sigma_2\sigma_1\sigma_2 = \sigma_2\sigma_1\sigma_2\sigma_1$ be the longest element in $W(\Theta)$.   
Considering the images of the two maps  $F_{\sigma_1\sigma_2\sigma_1\sigma_2 } (x_1, v_1, x_2, v_2) $ and $F_{\sigma_2\sigma_1\sigma_2\sigma_1} (w_1, y_1, w_2, y_2)$, with $x_1, x_2, y_1, y_2 \in \mathfrak{u}_{\alpha_1} \cong \RR$ and $v_1, v_2, w_1, w_2 \in \mathfrak{u}_{\alpha_2} \cong \RR^{1,q-2}$, we want to understand when 
$$x_{\alpha_1}(x_1) x_{\alpha_2}(v_1) x_{\alpha_1}(x_2) x_{\alpha_2}(v_2) = x_{\alpha_2}(w_1) x_{\alpha_1}(y_1) x_{\alpha_2}(w_2)x_{\alpha_1}(y_2). $$
Comparing the entries of the matrices we get the following set of equations. 
\begin{enumerate}
\item $ x_1 + x_2 = y_1 + y_2$
\item $v_1 + v_2 = w_1 + w_2$
\item $x_1(v_1 + v_2) + x_2 v_2 = y_1 w_2$, which is equivalent to $x_2v_1 = y_1 w_1 + y_2 (w_1+w_2)$, which is equivalent to $y_1w_2 + x_2v_1 = (y_1 + y_2)(w_1+w_2) = (x_1+x_2) (v_1+v_2)$
\item $y_1 q_J(w_2) = x_1 q_J(v_1+v_2) + x_2 q_J(v_2)$, which is equivalent to $y_1q_J(w_1) + y_2 q_J(w_1+w_2) = x_2 q_J(v_1)$
\end{enumerate} 

An explicit solution is given by 

\begin{enumerate}
\item $ y_1 = \frac{q_J(x_1(v_1+v_2) + x_2 v_2)}{x_1 q_J(v_1+v_2) + x_2 q_J(v_2)}$
\item $y_2 = \frac{x_1x_2 q_J(v_1)}{x_1 q_J(v_1+v_2) + x_2 q_J(v_2)} $
\item $w_2 = \frac{x_1 q_J(v_1+v_2) + x_2 q_J(v_2)}{q_J(x_1(v_1+v_2) + x_2 v_2)} (x_1(v_1+v_2) + x_2 v_2)$
\item $w_1= \frac{1}{q_J(x_1(v_1+v_2) + x_2 v_2)} ((x_1x_2 q_J(v_1+v_2)  + x_2^2 q_J(v_2))v_1 - x_1x_2 q_J(v_1)(v_1+v_2))$
\end{enumerate} 

It can be easily checked that these equations are well defined as soon as $x_1\geq 0$, $v_1 \in c_{\alpha_2} $, $x_2>0$ and $v_2 \in c^\circ_{\alpha_2}$. 
Moreover, if $x_1, x_2>0$, $v_1, v_2 \in c^\circ_{\alpha_2} $, then $y_1, y_2>0$, $w_1, w_2 \in c^\circ_{\alpha_2} $. (The only thing which is not immediate from the formulas and needs to be checked  is that $q_J(w_1)$ is always positive). 
In particular, the $\Theta$-positive semigroup $U_\Theta^{>0}$ is well defined. 

One can analyze the structure of $U_\Theta^{>0}$ and $U_\Theta^{\geq 0}$ more closely, and give a parametrization of $U_\Theta^{\geq 0}$ as the disjoint union of images of 16 maps, 
where every map is associated to reduced word in the Weyl group $W(\Theta)$. For details we refer the reader to \cite{Guichard_Wienhard_pos}.

With the above parametrization of $U_\Theta^{>0}$ one can also check directly that Conjecture~\ref{conj:nilpotent} is satisfied in this case: 
If $v = a + w \in \mathfrak{u}_{\alpha_1} \oplus \mathfrak{u}_{\alpha_2}$, with $a \in \RR_+$, and $w \in c^\circ_{\alpha_2} $, then $\exp(v) = F_{\sigma_2\sigma_1\sigma_2\sigma_1} (\frac{1}{3} w, \frac{3}{4} a, \frac{2}{3} w, \frac{1}{4}a ) \in U_\Theta^{>0}$.

\section{Positivity and Higher Teichm\"uller theory}
 Lusztig's total positivity and the positivity in Hermitian Lie group played an important role in recent developments in geometry and low dimensional topology, which led to what now is often called higher Teichm\"uller theory. 

Let $\Sigma_{g}$ be an oriented topological surface of genus $g \geq 2$, and let $\pi_1(\Sigma_g)$ be its fundamental group.\footnote{We restrict to the case of closed surface to simplify the discussion. For surfaces with punctures or boundary components appropriate analogous statements hold. }
Then the boundary $\partial \pi_1(\Sigma_{g}) $ naturally identifies with $\RR\PP^1$, and $\pi_1(\Sigma_g)$ acts on $\RR\PP^1$ preserving the orientation. 
%
%
%

The Teichm\"uller space $\mathcal{T} (\Sigma_g)$ of $\Sigma_g$ is the moduli space of marked conformal structures on $\Sigma_g$. By the Uniformization theorem, any such conformal structure can be realized by a unique hyperbolic structure, i.e.\ a unique metric of constant curvature $-1$. The holonomy of this hyperbolic structure gives a representation $\rho: \pi_1(\Sigma_g) \to \mathrm{PSL}(2,\RR)$, which is faithful with discrete image. In turn any faithful and discrete representation $\rho: \pi_1(\Sigma_g) \to \mathrm{PSL}(2,\RR)$ induces a hyperbolic structure on $\Sigma_g$, identifying $\Sigma_g$ with the quotient of the hyperbolic plane $\HH^2$ by $\rho(\pi_1(\Sigma_g))$. 

The subset of discrete and faithful representations 
$$\mathrm{Hom}_{df}(\pi_1(\Sigma_g), \mathrm{PSL}(2,\RR))/\mathrm{PSL}(2,\RR) \subset \mathrm{Hom}(\pi_1(\Sigma_g), \mathrm{PSL}(2,\RR))/\mathrm{PSL}(2,\RR)$$
is a union of two connected components, each of which is homeomorphic to the Teichm\"uller space of $\Sigma_g$. Representations in one component induce hyperbolic structures on $\Sigma_g$ with the same orientation as the given one, the representations in the other induce the opposite orientation on $\Sigma_g$. 
In particular, there is a connected component of the representation variety 
$\mathrm{Hom}(\pi_1(\Sigma_g), \mathrm{PSL}(2,\RR))/\mathrm{PSL}(2,\RR)$, which consists entirely of discrete and faithful representations, and identifies with the Teichm\"uller space $\mathcal{T} (\Sigma_g)$ \cite{Weil}. 

In the past 25 years, for two families of Lie groups $G$ -- split real forms and Hermitian type -- new connected components in $ \mathrm{Hom}(\pi_1(\Sigma_g), G)/G$ have been discovered, which also consist entirely of discrete and faithful representations and share several properties with representations in $\mathrm{Hom}_{df}(\pi_1(\Sigma_g), \mathrm{PSL}(2,\RR))/\mathrm{PSL}(2,\RR)$. These connected components are called {\em higher Teichm\"uller spaces}. 

The two known families of higher Teichm\"uller spaces are the {\em Hitchin components}, which are defined when $G$ is a split real semisimple Lie group, and the {\em space of maximal representations}, which are defined when $G$ is a Lie group of Hermitian type. These spaces have been defined and investigated by very different methods. It turned out that these two families in fact share many features, and in fact both families can be characterized in terms of positive structures. 

\subsection{Hitchin components}
The Hitchin component has originally been defined by Hitchin using methods from the theory of Higgs bundles \cite{Hitchin}. We do not give the original definition here, but a more elementary one. 
We assume that $G$ is an adjoint split real semisimple Lie group. 
Let $\mathfrak{g}_0$ be a principal three dimensional subalgebra of $\mathfrak{g}$ \cite{Kostantsl2}.  
Let $\pi_p: \mathrm{SL}(2,\RR) \to G$ be the associated Lie group homomorphism. Fix any discrete and faithful embedding $\iota: \pi_1(\Sigma_g) \to \mathrm{SL}(2,\RR)$. 

Set $\rho_H := \pi_p \circ \iota: \pi_1(\Sigma_g) \to G$. The Hitchin component 
$$
\mathcal{T}_H(\Sigma_g, G)\subset   \mathrm{Hom}(\pi_1(\Sigma_g), G)/G
$$
is defined to be the connected component of $\mathrm{Hom}(\pi_1(\Sigma_g), G)/G$ containing $\rho_H$. Note that $\mathcal{T}_H(\Sigma_g, \mathrm{PSL}(2,\RR)) = \mathcal{T} (\Sigma_g)$.

Hitchin proved that $\mathcal{T}_H(\Sigma_g, G) \cong \RR^{(2g-2) dim(G)}$. 
Labourie, and Fock and Goncharov \cite{Labourie_anosov, Fock_Goncharov} established several interesting geometric properties for representations in the Hitchin component, we just mention a few: 
\begin{enumerate}
\item Every representation in $\mathcal{T}_H(\Sigma_g, G)$ is discrete and faithful. 
\item For every $\gamma \in \pi_1(\Sigma_g) - \{ id\}$ the image $\rho(\gamma)$ is diagonalizable with distinct eigenvalues of the same sign. 
\item Every representation in $\mathcal{T}_H(\Sigma_g, G)$ is Anosov with respect to the Borel subgroup $B = P_\Delta$.\footnote{We refer the reader to \cite{Labourie_anosov, Guichard_Wienhard_DoD} for the definition of Anosov representations.}
\end{enumerate}

In particular Labourie, Guichard, Fock and Gonachrov established the following characterization. 
\begin{theorem}\cite{Labourie_anosov, Fock_Goncharov, Guichard_hitchin}
Let $\rho: \pi_1(\Sigma_g) \to G$ be a representation. Then $\rho \in \mathcal{T}_H(\Sigma_g, G)$ if and only if there exists a continuous $\rho$-equivariant map 
$\xi: \RR\PP^1 \to G/B$ which sends positive triples in $\RR\PP^1$ to positive triples in $G/B$, where the notion of positivity of triples in $G/B$ is given by Lusztig's total positivity (see Section~\ref{sec:triples}).  
\end{theorem}

We refer the reader to \cite{Labourie_crossratio, Labourie_McShane, Lee_Zhang, Guichard_Wienhard_duke, Guichard_Wienhard_DoD, Bridgeman_Canary_Labourie_Sambarino,Potrie_Sambarino} for more geometric and dynamical properties of Hitchin representations.

\subsection{Maximal representations}
The space of maximal representations is defined when $G$ is a Lie group of Hermitian type. In this case there is a characteristic number 
$$
\tau: \mathrm{Hom}(\pi_1(\Sigma_g), G)/G \to \frac{1}{e_G}\ZZ, 
$$
which is bounded in absolute value by $|\tau(\rho)| \leq |\chi(\Sigma_g)|  \mathrm{rk}(G)$, where $\mathrm{rk}(G)$ denotes the real rank of $G$, and $e_G$ is an integer depending on $G$. 

Maximal representations are those that saturate the upper bound: 
$$
\mathcal{T}_m(\Sigma_g, G) = \tau^{-1} (|\chi(\Sigma_g)|  \mathrm{rk}(G)). 
$$

Since $\tau$ is an integer-valued continuous map on $\mathrm{Hom}(\pi_1(\Sigma_g), G)/G$ it is immediate that $\mathcal{T}_m(\Sigma_g, G)$ is a union of connected components. 
Goldman proved that $\mathcal{T}_m(\Sigma_g, \mathrm{PSL}(2,\RR)) = \mathcal{T} (\Sigma_g)$ \cite{Goldman_thesis, Goldman_components}.

Maximal representations have several interesting properties, which have been proved in \cite{Burger_Iozzi_Wienhard_toledo, Burger_Iozzi_Labourie_Wienhard}
\begin{enumerate}
\item Any representation in $ \mathcal{T}_m(\Sigma_g, G)$ is discrete and faithful. 
\item Every representation in $\mathcal{T}_m(\Sigma_g, G)$ is Anosov with respect to the parabolic subgroup $ P_\Theta$, where $P_\Theta$ is the parabolic subgroup stabilizing a point in the Shilov boundary.  
\item For every $\gamma \in \pi_1(\Sigma_g) - \{ id\}$ the image $\rho(\gamma)$ as a unique attracting and a unique repelling fix point in the Shilov boundary; they are transverse to each other.  
\end{enumerate}

Maximal representations into Lie groups of Hermitian type which are not of tube type satisfy a rigidity theorem, which essentially reduces their study to the case of maximal representations into Hermitian symmetric spaces of tube type, \cite{Toledo, Hernandez, Burger_Iozzi_Wienhard_toledo}. 

So we assume now without loss of generality that $G$ is of tube type in order to simplify the statement of the characterization of maximal representations in terms of positivity: 

 \begin{theorem}\cite{Burger_Iozzi_Wienhard_toledo}
Let $\rho: \pi_1(\Sigma_g) \to G$ be a representation. Then $\rho \in \mathcal{T}_m(\Sigma_g, G)$ if and only if there exists a continuous $\rho$-equivariant map 
$\xi: \RR\PP^1 \to \check{S} = G/P_\Theta$ which sends positive triples in $\RR\PP^1$ to positive triples in the Shilov boundary $\check{S} = G/P_\Theta$, where the notion of positivity of triples in $\check{S}$ is given by the generalized Maslov index (see Section~\ref{sec:maslov}).
\end{theorem}

We refer the reader to \cite{Burger_Iozzi_Wienhard_toledo, Wienhard_mapping, Guichard_Wienhard_DoD, Burger_Pozzetti} for more details and further properties of maximal representations. 

\subsection{Positive representations}
Based on the common characterization of representations in the Hitchin components as well as maximal representations in terms of positive triples in flag varieties, we propose the following definition. 

\begin{definition}
Let $G$ be a semisimple Lie group with a $\Theta$-positive structure. 
A representation $\rho: \pi_1(\Sigma_{g})  \to G$ is said to be $\Theta$-positive if there exists a $\rho$-equivariant positive map 
$\xi: \partial \pi_1(\Sigma_{g})  = \RR\PP^1 \longrightarrow G/P_{\Theta}$ 
sending positive triples in $\RR\PP^1$ to $\Theta$-positive triples in $G/P_{\Theta}$. 
\end{definition}

We make the following Conjecture: 
\begin{conjecture}[Guichard-Labourie-Wienhard]\label{conj:GLW}
Let $\rho: \pi_1(\Sigma_{g})  \to G$ be a  $\Theta$ -positive representation. Then $\rho$ is $P_\Theta$-Anosov.  
The set of $\Theta$-positive representations $\rho: \pi_1(\Sigma_{g})  \to G$ is open and closed in $ \mathrm{Hom}(\pi_1(\Sigma_g), G)/G$.
\end{conjecture}

%
This conjecture will be addressed in forthcoming work of the authors with Fran\c{c}ois Labourie \cite{Guichard_Labourie_Wienhard}. 

A positive answer to the conjecture together with Theorem~\ref{thm:classification} implies 
\begin{corollary}
There are two other families of Lie groups $G$, namely $\SO(p,q)$, $p\neq q$, and the exceptional family modelled on $F_4$, which admit higher Teichm\"uller spaces $\mathcal{T}_{\Theta}(\Sigma_g, G) \subset \mathrm{Hom}(\pi_1(\Sigma_g), G)/G$.  
\end{corollary}

\subsection{Connected components of the representation variety} 
When $G$ is  a compact or a complex simple Lie group, then the connected components of $ \mathrm{Hom}(\pi_1(\Sigma_g), G)/G$ are in one to one correspondence with elements in $\pi_1(G)$, \cite{Atiyah_Bott, Goldman_conj, Li}. This is not true when $G$ is a real simple Lie group. In this case there can be additional connected components which are not distinguished by characteristic classes. 
This phenomenon happens in particular in the case when $G$ is a split real Lie group, and when $G$ is of Hermitian type. 
In the first case, the Hitchin components give rise to additional connected components. 
In the second case, the space of maximal representations, which is given as a level set of a characteristic class in $\mathrm{H}^2(\Sigma_g, \pi_1(G)) \cong \pi_1(G)$, splits into several connected components. These additional components can actually distinguished by additional topological invariants, which have been defined one the one hand using the theory of Higgs bundles in \cite{Bradlow_GP_Gothen_classical, Gothen, GP_Gothen_Mundet}, and using geometric consequences of the Anosov property of maximal representations in \cite{Guichard_Wienhard_top}.

Based on Conjecture~\ref{conj:GLW} we make the following conjecture 

\begin{conjecture}(Guichard-Wienhard)\label{conj:conncomp}
When $G$ carries a $\Theta$-positive structure, then there are additional connected components in $ \mathrm{Hom}(\pi_1(\Sigma_g), G)/G$, which are not distinguished by characteristic classes. 
\end{conjecture} 

In fact if Conjecture~\ref{conj:GLW} holds, then the construction of additional  topological invariants to Anosov representations from \cite{Guichard_Wienhard_top} can be used for $\Theta$-positive representations. These invariants together with the construction of explicit examples of $\Theta$-positive representations will then allow to give a lower bound on the number of connected components 
of $\Theta$-positive representations. 

In some special cases, when $G = \SO(n,n+1)$ Conjecture~\ref{conj:conncomp} has been partly confirmed in recent work of Collier \cite{Collier_thesis} and in ongoing work of Bradlow, Collier, Garcia-Prada, Gothen, and Oliveira using methods from the theory of Higgs bundles. 

%


%

\frenchspacing
\newcommand{\MR}[1]{}


\begin{thebibliography}{10}


\bibitem{Ando}
Tsuyoshi Ando, \emph{Totally positive matrices}, Linear Algebra Appl.
  \textbf{90} (1987), 165--219. \MR{884118}

\bibitem{Atiyah_Bott}
Michael~F. Atiyah and Raoul Bott, \emph{The {Y}ang-{M}ills equations over
  {R}iemann surfaces}, Philos. Trans. Roy. Soc. London Ser. A \textbf{308}
  (1983), no.~1505, 523--615. \MR{702806 (85k:14006)}

\bibitem{BBHIW2}
Gabi Ben~Simon, Marc Burger, Tobias Hartnick, Alessandra Iozzi, and Anna
  Wienhard, \emph{On order preserving representations}, J. London Math. Soc.
  (2), to appear.

\bibitem{BBHIW1}
\bysame, \emph{On weakly maximal representations of surface groups}, J.
  Differential Geom., to appear.

\bibitem{BenSimon_Hartnick}
Gabi Ben~Simon and Tobias Hartnick, \emph{Invariant orders on {H}ermitian {L}ie
  groups}, J. Lie Theory \textbf{22} (2012), no.~2, 437--463. \MR{2976926}

\bibitem{Benoist}
Yves Benoist, \emph{Automorphismes des c\^ones convexes}, Invent. Math.
  \textbf{141} (2000), no.~1, 149--193. \MR{1767272 (2001f:22034)}

\bibitem{Berenstein_Zelevinsky}
Arkady Berenstein and Andrei Zelevinsky, \emph{Total positivity in {S}chubert
  varieties}, Comment. Math. Helv. \textbf{72} (1997), no.~1, 128--166.
  \MR{1456321}

\bibitem{Bradlow_GP_Gothen_classical}
Steven~B. Bradlow, Oscar Garc{\'{\i}}a-Prada, and Peter~B. Gothen,
  \emph{Maximal surface group representations in isometry groups of classical
  {H}ermitian symmetric spaces}, Geom. Dedicata \textbf{122} (2006), 185--213.
  \MR{2295550 (2008e:14013)}

\bibitem{Bridgeman_Canary_Labourie_Sambarino}
Martin Bridgeman, Richard Canary, Fran{\c{c}}ois Labourie, and Andr\'es
  Sambarino, \emph{The pressure metric for {A}nosov representations}, Geom.
  Funct. Anal. \textbf{25} (2015), no.~4, 1089--1179. \MR{3385630}

\bibitem{Burger_Iozzi_Labourie_Wienhard}
Marc Burger, Alessandra Iozzi, Fran{\c{c}}ois Labourie, and Anna Wienhard,
  \emph{Maximal representations of surface groups: {S}ymplectic {A}nosov
  structures}, Pure Appl. Math. Q. \textbf{1} (2005), no.~3, Special Issue: In
  memory of Armand Borel. Part 2, 543--590. \MR{2201327 (2007d:53064)}

\bibitem{Burger_Iozzi_Wienhard_tight}
Marc Burger, Alessandra Iozzi, and Anna Wienhard, \emph{Tight homomorphisms and
  {H}ermitian symmetric spaces}, Geom. Funct. Anal. \textbf{19} (2009), no.~3,
  678--721. \MR{2563767}

\bibitem{Burger_Iozzi_Wienhard_toledo}
\bysame, \emph{Surface group representations with maximal {T}oledo invariant},
  Ann. of Math. (2) \textbf{172} (2010), no.~1, 517--566. \MR{2680425}

\bibitem{Burger_Pozzetti}
Marc Burger and Beatrice Pozzetti, \emph{Maximal representations, non
  {A}rchimedean {S}iegel spaces, and buildings}, preprint, arXiv:1509.01184.

\bibitem{Clerc}
Jean-Louis Clerc, \emph{L'indice de {M}aslov g\'en\'eralis\'e}, J. Math. Pures
  Appl. (9) \textbf{83} (2004), no.~1, 99--114.

\bibitem{Clerc_Orsted}
Jean-Louis Clerc and Bent {\O}rsted, \emph{The {M}aslov index revisited},
  Transform. Groups \textbf{6} (2001), no.~4, 303--320.

\bibitem{Collier_thesis}
Brian Collier, \emph{Finite order automorphisms of {H}iggs bundles: theory and
  application}, Ph.D. thesis, UIUC, 2016.

\bibitem{Fock_Goncharov}
Vladimir Fock and Alexander Goncharov, \emph{Moduli spaces of local systems and
  higher {T}eichm\"uller theory}, Publ. Math. Inst. Hautes \'Etudes Sci.
  (2006), no.~103, 1--211. \MR{2233852}

\bibitem{Fomin_ICM}
Sergey Fomin, \emph{Total positivity and cluster algebras}, Proceedings of the
  {I}nternational {C}ongress of {M}athematicians. {V}olume {II}, Hindustan Book
  Agency, New Delhi, 2010, pp.~125--145. \MR{2827788}

\bibitem{Gantmacher_Krein}
Felix {Gantmacher} and Mark {Krein}, \emph{{Sur les matrices oscillatoires.}},
  {C. R. Acad. Sci., Paris} \textbf{201} (1935), 577--579 (French).

\bibitem{GP_Gothen_Mundet}
Oscar Garc\'{\i}a-Prada, Peter~B. Gothen, and Ignasi Mundet~i Riera,
  \emph{Higgs bundles and surface group representations in the real symplectic
  group}, J. Topol. \textbf{6} (2013), no.~1, 64--118. \MR{3029422}

\bibitem{Goldman_thesis}
William~M. Goldman, \emph{Discontinuous groups and the {E}uler class}, Ph.D.
  thesis, University of California at Berkeley, 1980.

\bibitem{Goldman_conj}
\bysame, \emph{Geometric structures on manifolds and varieties of
  representations}, Geometry of group representations ({B}oulder, {CO}, 1987),
  Contemp. Math., vol.~74, Amer. Math. Soc., Providence, RI, 1988,
  pp.~169--198. \MR{957518 (90i:57024)}

\bibitem{Goldman_components}
\bysame, \emph{Topological components of spaces of representations}, Invent.
  Math. \textbf{93} (1988), no.~3, 557--607. \MR{952283 (89m:57001)}

\bibitem{Gothen}
Peter~B. Gothen, \emph{Components of spaces of representations and stable
  triples}, Topology \textbf{40} (2001), no.~4, 823--850. \MR{1851565
  (2002k:14017)}

\bibitem{Guichard_hitchin}
Olivier Guichard, \emph{Composantes de {H}itchin et repr\'esentations
  hyperconvexes de groupes de surface}, J. Differential Geom. \textbf{80}
  (2008), no.~3, 391--431. \MR{2472478 (2009h:57031)}

\bibitem{Guichard_Labourie_Wienhard}
Olivier Guichard, Fran\c{c}ois Labourie, and Anna Wienhard, \emph{Positive
  representations}, in preparation, 2016.

\bibitem{Guichard_Wienhard_duke}
Olivier Guichard and Anna Wienhard, \emph{Convex foliated projective structures
  and the {H}itchin component for {${\rm PSL}\sb 4({\bf R})$}}, Duke Math. J.
  \textbf{144} (2008), no.~3, 381--445. \MR{2444302}

\bibitem{Guichard_Wienhard_top}
\bysame, \emph{Topological invariants of {A}nosov representations}, J. Topol.
  \textbf{3} (2010), no.~3, 578--642. \MR{2684514}

\bibitem{Guichard_Wienhard_DoD}
\bysame, \emph{Anosov representations: domains of discontinuity and
  applications}, Invent. Math. \textbf{190} (2012), no.~2, 357--438.
  \MR{2981818}

\bibitem{Guichard_Wienhard_pos}
\bysame, \emph{$\theta$-positivity}, in preparation, 2016.

\bibitem{Helgason}
Sigurdur Helgason, \emph{Differential Geometry, Lie Groups, and Symmetric Spaces}, Graduate Studies in Mathematics vol. 34, AMS 2012. 

\bibitem{Hernandez}
Luis {H}ern\`andez, \emph{Maximal representations of surface groups in bounded
  symmetric domains}, Trans. Amer. Math. Soc. \textbf{324} (1991), 405--420.
  \MR{1033234 (91f:32040)}

\bibitem{Hilgert_Neeb}
Joachim Hilgert and Karl-Herrmann Neeb, \emph{Lie Semigroups and their Applications}, 
Lecture Notes in Mathematics, Vol. 1552, 1993.

\bibitem{Hitchin}
Nigel~J. Hitchin, \emph{Lie groups and {T}eichm\"uller space}, Topology
  \textbf{31} (1992), no.~3, 449--473. \MR{1174252 (93e:32023)}

\bibitem{Kaneyuki}
Soji Kaneyuki, \emph{On the causal structures of the \v {S}ilov boundaries of
  symmetric bounded domains}, Prospects in complex geometry ({K}atata and
  {K}yoto, 1989), Lecture Notes in Math., vol. 1468, Springer, Berlin, 1991,
  pp.~127--159. \MR{1123540}

\bibitem{Karlin}
Samuel Karlin, \emph{Total positivity. {V}ol. {I}}, Stanford University Press,
  Stanford, Calif, 1968. \MR{0230102}

\bibitem{Kostantsl2}
Bertram Kostant, \emph{The principal three-dimensional subgroup and the {B}etti
  numbers of a complex simple {L}ie group}, Amer. J. Math. \textbf{81} (1959),
  973--1032. \MR{0114875 (22 \#5693)}

\bibitem{Labourie_anosov}
Fran{\c{c}}ois Labourie, \emph{Anosov flows, surface groups and curves in
  projective space}, Invent. Math. \textbf{165} (2006), no.~1, 51--114.
  \MR{2221137}

\bibitem{Labourie_crossratio}
\bysame, \emph{Cross ratios, surface groups, {${\rm PSL}(n,{\bf R})$} and
  diffeomorphisms of the circle}, Publ. Math. Inst. Hautes \'Etudes Sci.
  (2007), no.~106, 139--213. \MR{2373231}

\bibitem{Labourie_McShane}
Fran{\c{c}}ois Labourie and Gregory McShane, \emph{Cross ratios and identities
  for higher {T}eichm\"uller-{T}hurston theory}, Duke Math. J. \textbf{149}
  (2009), no.~2, 279--345. \MR{2541705}

\bibitem{Lee_Zhang}
Gye-Seon Lee and Tengren Zhang, \emph{Collar lemma for {H}itchin
  representations}, preprint, arXiv:1411.2082.

\bibitem{Li}
Jun Li, \emph{The space of surface group representations}, Manuscripta Math.
  \textbf{78} (1993), no.~3, 223--243. \MR{1206154 (94c:58022)}

\bibitem{Loewner}
Charles Loewner, \emph{On totally positive matrices}, Math. Z. \textbf{63}
  (1955), 338--340. \MR{0073657}

\bibitem{Lusztig}
George Lusztig, \emph{Total positivity in reductive groups}, Lie theory and
  geometry, Progr. Math., vol. 123, Birkh\"auser Boston, Boston, MA, 1994,
  pp.~531--568. \MR{1327548 (96m:20071)}

\bibitem{Potrie_Sambarino}
Rafael Potrie and Andr\'es Sambarino, \emph{Eigenvalues and {E}ntropy of a
  {H}itchin representation}, preprint, arXiv:1411.5405.

\bibitem{Schoenberg}
Isac Schoenberg, \emph{\"{U}ber variationsvermindernde lineare
  {T}ransformationen}, Math. Z. \textbf{32} (1930), no.~1, 321--328.
  \MR{1545169}

\bibitem{Toledo}
Domingo Toledo, \emph{Representations of surface groups in complex hyperbolic
  space}, J. Differential Geom. \textbf{29} (1989), no.~1, 125--133. \MR{978081
  (90a:57016)}

\bibitem{Weil}
Andr{\'e} Weil, \emph{On discrete subgroups of {L}ie groups. {II}}, Ann. of
  Math. (2) \textbf{75} (1962), 578--602. \MR{0137793}

\bibitem{Whitney}
Anne Whitney, \emph{A reduction theorem for totally positive matrices}, J.
  Analyse Math. \textbf{2} (1952), 88--92. \MR{0053173}

\bibitem{Wienhard_mapping}
Anna Wienhard, \emph{The action of the mapping class group on maximal
  representations}, Geom. Dedicata \textbf{120} (2006), 179--191. \MR{2252900
  (2008g:20112)}

\bibitem{Yokota}
Ichiro Yokota, \emph{Exceptional {L}ie groups}, preprint, arXiv:0902.0431.

\end{thebibliography}

\providecommand{\bysame}{\leavevmode\hbox to3em{\hrulefill}\thinspace}

\end{document}